\documentclass[a4paper,10pt]{article}
\usepackage[margin=2.5cm]{geometry}
\usepackage[utf8]{inputenc}
\usepackage{amsmath,amsthm,amssymb}

\usepackage{hyperref}
\hypersetup{
  colorlinks=true,
  linkcolor=black,
  citecolor=blue,
  urlcolor=blue,
  pdftitle={},
  pdfauthor={},
}
\newtheorem{thm}{Theorem}[section]
\newtheorem{lem}[thm]{Lemma}
\newtheorem{rem}[thm]{Remark}
\newtheorem{cor}[thm]{Corollary}
\newtheorem{definition}[thm]{Definition}
\swapnumbers
\usepackage{xcolor}

\newcommand{\eps}{\varepsilon}
\newcommand{\R}{\mathbb{R}}
\newcommand{\N}{\mathbb{N}}

\newcommand{\io}{\int_\Omega}

\usepackage{graphicx}

\allowdisplaybreaks
\title{Emergence of large densities in a chemotaxis system with signal-dependent motility}
\author{Max Hegenbart\\
\small{Leibniz Universität Hannover}\\
\small{Institut für Angewandte Mathematik}\\
\small{hegenbart@ifam.uni-hannover.de}}

\begin{document}

\maketitle

\begin{abstract}
In this paper, we focus on the emergence of large densities of a cell population in finite time in
the following parabolic-elliptic Keller-Segel system
with signal-dependent motility and a source term of logistic type
\begin{equation}\label{Hauptsystem Abstract}
\begin{cases}
u_{\eps t} = \Delta(e^{-v_\eps}u_\eps) + \eps(\kappa u_\eps - \mu u_\eps^2)
,\\
0 =\Delta v_{\eps}-v_\eps+ u_{\eps}.
\end{cases}
\end{equation}
in a smooth bounded domain $\Omega\subseteq\R^2$ with homogenuous Neumann boundary conditions
and parameters $\kappa,\mu>0$ and $\eps\in(0,1)$.
We will show that the solution $u_\eps$ can exceed any certain value $M>0$ before a time $T(M)>0$ under the
condition that $\eps>0$ is sufficiently small.
To prove the result, we show convergence of classical
solutions of \eqref{Hauptsystem Abstract} to classical solutions of the
corresponding limit system as $\eps\searrow 0$
\begin{equation}\label{Grenzsystem Abstract}
 \begin{cases}
  u_t=\Delta(e^{-v}u),\\
  -\Delta v+v=u.
 \end{cases}
\end{equation}
 The main challenge is to find suitable bounds for $u_\eps$ and $v_\eps$, that are $\eps$-independent, in Lebesgue and Hölder spaces, whereas for fixed $\eps>0$ global existence and boundedness is known.
\end{abstract}
\textbf{Math Subject Classification (MSC 2020):} 35K55; 
35B45;   	
35B36;   	
35Q92;   	
35B40; 
92C17   	

\section{Introduction}
In the context of chemotaxis, the emergence of patterns
(when a cell population is highly concentrated in some regions)
plays an important role in biology, health care etc.
In the mathematical analysis, the usual way to search for such patterns revolves around a study of the
solution in the limit of its existence time, i.e. the long-term or, possibly, blow-up behaviour
(with results often dealing with the question of
global existence vs. finite-time blow-up or boundedness vs. infinite-time blow-up of solutions).
In this work, however, we are interested in the emergence of arbitrarily large
population densities before some possibly arising chemotactic collapse and within finite time,
that is, on intermidiate time scales.
To this end, we study the model
\begin{equation}\label{Hauptsystem}
\begin{cases}
u_{\eps t} = \Delta(e^{-v_\eps}u_\eps) + \eps(\kappa u_\eps - \mu u_\eps^2)
& \text{in } \Omega \times (0,T_{\max,\eps}), \\[0.5em]

0 = \Delta v_{\eps} - v_{\eps}+u_{\eps}
& \text{in } \Omega \times (0,T_{\max,\eps}), \\[0.5em]

\partial_\nu u_\eps = \partial_\nu v_\eps = 0
& \text{on } \partial\Omega \times (0,T_{\max,\eps}), \\[0.5em]

u_\eps(x,0) = u_0(x)
& \text{in } \Omega,
\end{cases}
\end{equation}
for $\eps\in(0,1)$ and a smooth bounded domain $\Omega\subseteq\R^2$,
which describes signal-dependent motility of
cells. The functions $u=u(x,t)$ and $v=v(x,t)$ denote the population density or the signal concentration,
respectively.
This kind of motility can, for instance, be observed in colonies of \textit{Escherichia coli},
which are bacteria that humans have in their guts and that play an important role
during digestion \cite{Berg2004}. Here, the motility depends on the density of the
substance \textit{acyl-homoserine lactone}  which acts as chemoattractant signal \cite{Liu2011-ae}.
In regions of low density, bacteria are motile, where as in regions of large signal density, bacteria are
essentially immotile (cf. \cite{Liu2011-ae,PhysRevLett.108.198102}).
The logistic source term $\eps(\kappa u_\eps-\mu u_\eps^2)$
models the birth and death of cells.
If we set $\eps=0$, we denote by $(u,v)$ the solution of
\begin{equation}\label{eps=0 System}
\begin{cases}
u_t = \Delta(e^{-v}u)
& \text{in } \Omega \times (0,T_{\max}), \\[0.5em]

-\Delta v + v = u
& \text{in } \Omega \times (0,T_{\max}), \\[0.5em]

\partial_\nu u = \partial_\nu v = 0
& \text{on } \partial\Omega \times (0,T_{\max}), \\[0.5em]

u(x,0) = u_0(x)
& \text{in } \Omega.
\end{cases}
\end{equation}
In comparison, the classical Keller-Segel
system \cite{KELLER1970399,MR2013508,BELBELTAOWIN2015,MR4071449} is known to be
\begin{equation}\label{klassisches Keller-Segel}
 \begin{cases}
  u_t=\Delta u-\nabla\cdot(u\nabla v),\\
  v_t=\Delta v-v+u,\\
  \partial_\nu u=\partial_\nu v=0,\\
  u(x,0)=u_0(x),v(x,0)=v_0(x).
 \end{cases}
\end{equation}
In order to see the relation between \eqref{Hauptsystem} and \eqref{eps=0 System}
to typical chemotaxis equations as in \eqref{klassisches Keller-Segel}, we
expand the Laplacian then in the first equation, obtaining
\begin{align*}
 u_t=\Delta(e^{-v}u)=\nabla\cdot(e^{-v}\nabla u-e^{-v}u\nabla v).
\end{align*}
In comparison, for the classical Keller-Segel system \eqref{klassisches Keller-Segel} we have
\begin{align*}
 u_t=\nabla\cdot(\nabla u-u\nabla v).
\end{align*}
Hence, the structure of the parabolic equation in \eqref{Hauptsystem} and \eqref{eps=0 System}
is the same as in \eqref{klassisches Keller-Segel}, but the so-called \textit{motility function} $e^{-v}$
controls the strength of diffusion and chemotaxis (cf. \cite{Liu2011-ae,PhysRevLett.108.198102}).
Unboundedness of solutions in finite time is known to occur in
Keller-Segel systems for some initial data (see \cite{MR1627338,Winkler_2013} for fully-parabolic,
and \cite{Jager_Luckhaus_1992,MR1931303} for parabolic-elliptic systems),
whereas the addition of logistic terms guarantees global
solutions under certain conditions (see \cite{Osaki} for global existence or
\cite{TelloWinkler_2007,Winkler2010} for bounded solutions).
These statements describe the behaviour of solutions close to their maximal existence time.
Here we are interested in whether the cell density may become arbitrarily large in some finite time if
we allow a weak effect of a logistic source term compared to diffusion and chemotaxis, as is known to be
the case for \eqref{klassisches Keller-Segel} \cite{MR3651259}.
Therefore, we consider \eqref{Hauptsystem} with small values for $\eps>0$.
Due to \cite{FUJIE20205338} we know that solutions are bounded (for more details see the discussion below).
Hence, any kind of blow-up cannot occur. Nevertheless, they can exceed any value for sufficiently small
values of $\eps$.
Our strategy is to use the $(\eps=0)$-system \eqref{eps=0 System},
for which we will show that it is not just a ``formal'' limit system,
but solutions of \eqref{eps=0 System} can actually be obtained as limits of
solutions to \eqref{Hauptsystem} in a suitable sense (compatible with transferring pointwise largeness
and fulfillment of the equation). The strategy to use a limit system, where blow-up may occur,
and prove uniform convergence to solutions of it
can also be found in \cite{WINK2014,Lankeit2015} for logistic growth or in \cite{Xu2020,Kohatsu_Lankeit2024}
for two-species-models
(there the diffusion is weak compared to chemotaxis and growth terms).
Unboundedness in $L^p(\Omega)$ for all $p>\frac{n}{2}$ was shown for \eqref{eps=0 System} based on blow-up
in the elliptic-parabolic version for the motility function $\gamma(v)=v^{-k}$ with sufficiently large $k$
\cite{wang2024singular}.\\
Let us briefly summarize some known results for chemotaxis systems with signal-dependent motility like
\eqref{Hauptsystem} and \eqref{eps=0 System}. For our model, we assume the following:
\begin{align}\label{Generalvoraussetzungen}
 \begin{split}
  &\Omega\subseteq\R^2\text{ be a smooth bounded domain,}\\
  &\kappa,\mu>0,\\
  &u_0\in W^{1,q}(\Omega)\text{ for }q>2,\quad u_0\geq 0\:\text{in }\overline{\Omega},\quad u_0\not\equiv 0,\\
  &v_0:=(I-\Delta)^{-1}u_0.
 \end{split}
\end{align}
In general there are more options for choosing motility functions in \eqref{Hauptsystem}
than $\gamma(v)=e^{-v}$.
A typical assumption is that such a function $\gamma$ satisfies
\begin{align}\label{Wachstumsbedingung an gamma}
 \gamma\in C^3([0,\infty)),\quad\gamma>0\text{ in }[0,\infty),\quad\gamma'<0\text{ in }[0,\infty),
\end{align}
as in \cite{jin2018boundedness,FUJIE20205338,fujie2022global}.
To summarize the known results,
the consideration of signal-dependent motility by including a motility function
satisfying \eqref{Wachstumsbedingung an gamma} prevents finite-time blow-up.
Whether solutions are bounded or blow up in infinite time depends on the certain choice
of the motility function (for infinite-time blow-up see \cite{FUJIE20205338} for the parabolic-elliptic, and
\cite{MR4342933,MR4143400,fujie2022global} for the fully parabolic setting; boundedness is proven in
\cite{ahn2019global,FUJIE20205338,MR4452315}). The addition of a logistic term excludes
unboundedness in \eqref{Hauptsystem} if the motility function satisfies \eqref{Wachstumsbedingung an gamma}
(see \cite{FUJIE20205338,MR4478289} for parabolic-elliptic, and
\cite{jin2018boundedness,MR3899907,lv2021n,tao2022global} for fully parabolic systems).
Moreover, for large values of $\mu$ (cf. \cite{FUJIE20205338,MR3944389}) or if $\gamma$ satisfies a certain
ordinary differential equation (ODI) (cf. \cite{desvillettes2023weak,MR4874468}),
convergence of solutions to constant steady-states in $L^\infty$ or $L^2$, respectively, is known.
The strategy of proving blow-up in system \eqref{eps=0 System} (which has been used in \cite{FUJIE20205338})
is the same as for the classical Keller-Segel system,
since the systems \eqref{klassisches Keller-Segel} and \eqref{eps=0 System} share the same Lyapunov functional
\begin{align*}
 \mathcal{F}(u,v):=\io u\ln(u)-\io uv+\frac{1}{2}\io v^2+\frac{1}{2}\io\vert\nabla v\vert^2.
\end{align*}
As mentioned above, in order to prove the emergence of large densities in \eqref{Hauptsystem} our aim is
to obtain, possibly along a subsequence, uniform convergence of solutions of \eqref{Hauptsystem}
to solutions of \eqref{eps=0 System}.
The convergence statement then is:
\begin{thm}\label{Theorem Konvergenzen}
Let $\Omega\subseteq\R^2$ be a smooth bounded domain, $\kappa,\mu>0$, and $(\eps_{j})_{j\in\N}\subseteq(0,1)$
be a nullsequence. Then there is a subsequence $\eps_{j_k}\searrow 0$ such that
the corresponding unique classical solutions $((u_{\eps_{j_k}},v_{\eps_{j_k}}))_{k\in\N}$
of \eqref{Hauptsystem} satisfy
 \begin{align*}
  &u_{\eps_{j_k}}\to u\quad\text{in }C_{loc}^0(\overline{\Omega}\times[0,\infty)),\\
  &v_{\eps_{j_k}}\to v\quad\text{in }C_{loc}^{2,0}(\overline{\Omega}\times[0,\infty)),
 \end{align*}
 where the pair
 \begin{align*}
  (u,v)\in\left( C^0(\overline{\Omega}\times[0,\infty))
 \cap C^{2,1}(\overline{\Omega}\times(0,\infty)) \right)^2,
 \end{align*}
 solves \eqref{eps=0 System} in the classical sense.
\end{thm}
The strategy will be to derive a local-in-time $L^\infty$-bound for $u_\eps$.
The computations will be similar to those
from \cite{FUJIE20205338}, but here we have to take care that our bounds are independent of $\eps$. Based on
parabolic regularity theory \cite{DiBenedetto1986}, we will show boundedness of $u_\eps$ in a Hölder space,
from which we can conclude uniform convergence. The convergence result for $v_\eps$ is based on elliptic
regularity theory, where the most important estimates for us will be stated in
Lemma \ref{elliptische Regularitätstheorie}. The uniform convergence will also be a useful tool to show
that the limits are solutions of \eqref{eps=0 System}. Due to Theorem \ref{Theorem Konvergenzen}, we can show
that solutions become large, although they are bounded, if we use suitable inital data.
\begin{thm}\label{Theorem große Dichten}
 Let $\Omega=B_R(0)\subseteq\R^2$ for $R\in(0,\infty)$, $\kappa\geq0$ and $\mu>0$.
 For each $\Lambda\in(8\pi,\infty)\setminus4\pi\N$, there is $u_0\in W^{1,q}(\Omega)$
 with $q>2$ such that $u_0$ is radially symmetric and satisfies $\Vert u_0\Vert_{L^1(\Omega)}=\Lambda$
and for each $M>0$ there are $T(M)>0$ and $\tilde{\eps}(M)\in(0,1)$ with the property that for any $\eps\in(0,\tilde{\eps}(M))$ one can find
$t_\eps\in(0,T(M))$ and $x_\eps\in\Omega$ such that the solution $(u_\eps,v_\eps)$ of \eqref{Hauptsystem} satisfies $u_\eps(x_\eps,t_\eps)>M$.
\end{thm}
This result shows that signal-dependent motility is capable of forming structure (at least in the sense
of arriving at high concentrations) within finite time, even if logistic population dynamics are included
in the system, which usually favour boundedness of solutions by (either their initial value or) the
carrying capacity of the domain.
The proof is mainly based on the known infinite time blow-up result from \cite[Theorem 1.3]{FUJIE20205338}.
An important tool will be a local-in-time upper bound for $v_\eps$,
which will give us a lower bound for $e^{-v_\eps}$. To show that classical solutions
of \eqref{eps=0 System} are limits of classical solutions of \eqref{Hauptsystem}, we introduce the concept
of $W^{1,2}$-solutions in Definition \ref{Stark schwache Lösungsdefinition}. Parabolic
regularity theory will then, in a subsequence step,
yield that the limit functions are even classical solutions.

\section{Preliminaries}
Let us state the local existence and uniqueness result for solutions of
\eqref{Hauptsystem} and \eqref{eps=0 System}.
It can be shown by a fixed point argument, where parabolic \cite{Lieberman1996, LSU}
and elliptic regularity theory \cite{GilbargTrudinger2001} are useful.
The case $\eps=0$ is proven by \cite[Lemma 3.1]{ahn2019global},
while for $\eps>0$ one can use some ideas from the proof for the fully-parabolic system
\cite[Lemma 2.1]{jin2018boundedness}.
\begin{lem}\label{lokale Existenz}
 Suppose \eqref{Generalvoraussetzungen}.
 Then there exists $T_{max,\eps}
 \in(0,\infty]$ such that system \eqref{Hauptsystem} permits a unique nonnegative, classical solution
 \begin{align*}
  (u_\eps,v_\eps)\in\left( C^0(\overline{\Omega}
 \times[0,T_{max,\eps}))\cap C^{2,1}(\overline{\Omega}\times(0,T_{max,\eps}))\right)^2.
 \end{align*}
 If $T_{max,\eps}<\infty$, then
 \begin{align*}
  \lim_{t\nearrow T_{max,\eps}}\Vert u_{\eps}(\cdot,t)\Vert_{L^\infty(\Omega)}=\infty.
 \end{align*}
\end{lem}
\begin{proof}
We set $R:=\Vert u_0\Vert_{L^\infty(\Omega)}+1$ and choose $T\in(0,1)$.
In the Banach space $X:=C^0(\overline{\Omega}\times[0,T])$ we
define the nonempty, closed and convex set $S_T:=\{ \tilde{u}_\eps\in X\mid
\Vert \tilde{u}_\eps(\cdot,t)\Vert_{L^\infty(\Omega)}\leq R \}$.
Then by \cite[Theorem 8.34]{GilbargTrudinger2001} (only stated for the Dirichlet case there), for any
$\tilde{u}_\eps$ there is a unique $\tilde{v}_\eps(\cdot,t)\in C^{1+\theta_1}(\overline{\Omega})$
with $\theta_1\in(0,1)$ and $t\in(0,T)$ such that $\tilde{v}_\eps$ solves
\begin{align}\label{elliptische Gleichung in Existenz und Eindeutigkeit}
 -\Delta\tilde{v}_\eps+\tilde{v}_\eps=\tilde{u}_\eps\text{ in }\Omega\times(0,T),\quad
 \partial_\nu\tilde{v}_\eps=0\text{ on }\partial\Omega\times(0,T).
\end{align}

Now, we define an operator
$\Phi:S_T\to X,\:\tilde{u}_\eps\mapsto u_\eps$, where
$u_\eps\in C^{\theta_2,\frac{\theta_2}{2}}(\overline{\Omega}\times[0,T])$ is a solution to
\begin{equation*}
 \begin{cases}
  u_{\eps t}=\Delta(u_\eps e^{-\tilde{v}_\eps})+\eps(\kappa u_\eps-\mu u_\eps\tilde{u}_\eps)&
  \text{in }\Omega\times(0,T),\\
  \partial_\nu u_\eps=0 &\text{on }\partial\Omega\times(0,T),\\
  u_\eps(x,0)=u_0(x)&\quad\text{in }\Omega\times\{0\},
 \end{cases}
\end{equation*}
by \cite[Theorem V.1.1]{LSU} for some $\theta_2\in(0,1)$.
Then we can show that $\Phi:S_T\to S_T$ is continuous and, moreover, by
the Arzel\`{a}-Ascoli theorem we obtain that $\Phi$ is compact. Hence, by Schauder's fixed point theorem we
obtain a local solution $u_\eps$ of the parabolic equation in \eqref{Hauptsystem}.
Using this solution, we can apply elliptic regularity theory and obtain
$v_\eps\in C^{2+\theta_2,\frac{\theta_2}{2}}(\overline{\Omega}\times[0,T])$, where $v_\eps$ is the unique
solution of \eqref{elliptische Gleichung in Existenz und Eindeutigkeit} with right-hand side $u_\eps$.
By \cite[Theorem IV.5.3]{LSU} we obtain
$u_\eps\in C^{2+\theta_2,\frac{2+\theta_2}{2}}(\overline{\Omega}\times[0,T])$. Hence, $(u_\eps,v_\eps)$ is
a local classical solution of \eqref{Hauptsystem}.
Uniqueness follows by a Gronwall argument. The extensibility criterion results from the
space used during the fixed point construction.
That $u_\eps,v_\eps$ are nonnegative follows by a comparison argument or
the Hopf Lemma, respectively.
\end{proof}

\begin{rem}
 Due to \cite{FUJIE20205338}, we can directly conclude $T_{max,\eps}=\infty$ for all $\eps\in(0,1)$.
\end{rem}

One difference to chemotaxis systems without a logistic source term is that in \eqref{Hauptsystem}
the total mass of cells is not constant anymore.
Due to an ODI comparison argument we can show that the total mass of is still bounded. Moreover, since
the second equation in \eqref{Hauptsystem} is elliptic, the total mass of the signal is
the same as the total mass of the cells.
\begin{lem}\label{Beschränktheit von u in L^1}
Suppose \eqref{Generalvoraussetzungen}. Then for all $\eps\in(0,1)$ and $t\in[0,\infty)$ the
corresponding classical solution $(u_\eps,v_\eps)$ of \eqref{Hauptsystem} satisfies
 \begin{align*}
  \io u_\eps(\cdot,t)\leq\max\left\{ \io u_0,\frac{\kappa}{\mu}\vert\Omega\vert \right\}=: m.
 \end{align*}
\end{lem}
\begin{proof}
We define $y_\eps(t):=\io u_\eps(\cdot,t)$. Then by Hölder's inequality we obtain
\begin{align*}
 y_\eps'(t)=\frac{d}{dt}\io u_\eps(\cdot,t)=\io u_{\eps t}(\cdot,t)
 &=\io \Delta(u_\eps e^{-v_\eps})(\cdot,t)+\eps\kappa\io u_\eps(\cdot,t)-\eps\mu\io u_\eps^2(\cdot,t)\\
 &\leq\eps\kappa\io u_\eps(\cdot,t)-\eps\mu\vert\Omega\vert\left( \io u_\eps(\cdot,t) \right)^2,
\end{align*}
for all $t\in(0,\infty)$ and $\eps\in(0,1)$,
where the first integral can be neglected due to the homogenuous boundary conditions.
Hence, $y_\eps(t)$ satisfies the ODI $\frac{1}{\eps}y_\eps'(t)\leq\kappa y(t)-\mu\vert\Omega\vert y_\eps^2(t)$
for all $t\in[0,\infty)$, which has the bounded supersolution $m$.
\end{proof}
A direct consequence of the ellipticity of the second equation in \eqref{Hauptsystem}
is that $u_\eps$ and $v_\eps$ share the same total mass.
By Lemma \ref{Beschränktheit von u in L^1} it follows that the total mass of the signal is also bounded.
\begin{lem}\label{u=v in L^1}
Suppose \eqref{Generalvoraussetzungen}. Then for all $\eps\in(0,1)$ and $t\in(0,\infty)$ the corresponding
classical solution $(u_\eps,v_\eps)$ of \eqref{Hauptsystem} satisfies
 \begin{align*}
  \io u_\eps(\cdot,t)=\io v_\eps(\cdot,t)\quad\forall t\in(0,\infty).
 \end{align*}

\end{lem}
\begin{proof}
 Since $\partial_\nu v=0$ on $\partial\Omega\times(0,\infty)$,
 integrating the second equation of \eqref{Hauptsystem} over $\Omega$ and applying
 the Gaussian divergence theorem yields
 \begin{align*}
  \io u_\eps(\cdot,t)=-\io\Delta v_\eps(\cdot,t)+\io v_\eps(\cdot,t)
  =-\int_{\partial\Omega}\nabla v_\eps(\cdot,t)\cdot\nu
  +\io v_\eps(\cdot,t)=\io v_\eps(\cdot,t),
 \end{align*}
for all $t\in(0,\infty)$ and $\eps\in(0,1)$ due to the homogenuous boundary conditions.
\end{proof}

Since solutions are nonnegative for $u_0\geq0$, we obtain two important estimates for the logistic term
in the first equation of \eqref{Hauptsystem}. They will be useful to find $\eps$-independent
$L^p$-bounds for the solutions of \eqref{Hauptsystem}.
\begin{lem}\label{Lemma Abschätzung Logistik}
 \begin{enumerate}
  \item Let $f:[0,\infty)\to\R,\: x\mapsto\kappa x^p-\mu x^{p+1}$ for $\kappa,\mu>0$ and $p\geq 1$.
  Then
  \begin{align*}
   \max_{x\geq 0}f(x)<\kappa\left( \frac{\kappa}{\mu} \right)^p.
  \end{align*}
  \item Let $g:[0,\infty)\to\R$ be defined as
  \begin{equation*}
        g(x):=\left\{\begin{array}{cc}
            0,  &x=0,  \\
            \kappa x\ln x-\mu x^2\ln x,& x>0.
        \end{array} \right.
    \end{equation*}
Then there is $C>0$ such that
\begin{align*}
 g(x)\leq C\quad\forall x\geq 0.
\end{align*}
 \end{enumerate}
\end{lem}
\begin{proof}
 Computing $f'$, $g'$ and the maxima of $f,g$ yield the result.
\end{proof}

We recall two important estimates based on elliptic regularity theory, which are also stated in
\cite[Lemma 2.3]{Lankeit2015}.
\begin{lem}\label{elliptische Regularitätstheorie}
 For $q\geq 1$, $\alpha>0$, there exists a constant $C>0$ such that any (classical) solution $v$ of
 \begin{align*}
  -\Delta v+v=f\quad\text{in }\Omega,\quad\partial_\nu v=0\quad\text{on }\partial\Omega
 \end{align*}
satisfies
\begin{align*}
 &\Vert v\Vert_{W^{2,q}(\Omega)}\leq C\Vert f\Vert_{L^q(\Omega)},\\
 &\Vert v\Vert_{C^{2+\alpha}(\Omega)}\leq C\Vert f\Vert_{C^\alpha(\Omega)},
\end{align*}
for any $f\in L^q(\Omega)$ or $f\in C^{\alpha}(\Omega)$, respectively.
\end{lem}

As stated before, to find an upper bound for $v_\eps$ is important to find a lower bound for $e^{-v_\eps}$.
The following lemma and proof are taken from \cite[Lemma 3.2]{FUJIE20205338}.
\begin{lem}\label{Zeitabhängige obere Schranke für v}
Suppose \eqref{Generalvoraussetzungen}. Then there is $C>0$ such that for all $(x,t)\in\Omega\times[0,\infty)$
and $\eps\in(0,1)$ the corresponding classical solution $(u_\eps,v_\eps)$ of \eqref{Hauptsystem} satisfies
 \begin{align*}
  v_{\eps}(x,t)\leq Ce^{(1+\kappa)t}.
 \end{align*}
\end{lem}
\begin{proof}
We plug $u_{\eps t}=(-\Delta v_\eps+v_{\eps})_t$ into the parabolic equation in \eqref{Hauptsystem} and
apply $(I-\Delta)^{-1}$. Since $(I-\Delta)^{-1}(\mu u_\eps^2)\geq 0$ by the maximum principle, we obtain
after integrating over $(0,t)$ for $t<\infty$
\begin{align*}
 v_\eps(x,t)\leq v_0(x)+\int_0^t (I-\Delta)^{-1}(e^{-v_\eps(x,s)}u_\eps(x,s)+\eps\kappa u_\eps(x,s))\:ds
 \quad\text{for all }x\in\Omega.
\end{align*}
Testing the elliptic equation of \eqref{Hauptsystem} by $e^{-v_\eps}$ and applying $(I-\Delta)^{-1}$, we have
\begin{align*}
 0\leq (I-\Delta)^{-1}(e^{-v_\eps}u_\eps)\leq(I-\Delta)^{-1}(u_\eps)=v_\eps
 \quad\text{in }\Omega\times(0,\infty).
\end{align*}
This implies (since $\eps\in(0,1)$)
\begin{align*}
 v_\eps(x,t)\leq v_0(x)+(1+\eps\kappa)\int_0^t v_\eps(x,s)\:ds
 \leq v_0(x)+(1+\kappa)\int_0^t v_\eps(x,s)\:ds\quad\text{for all }(x,t)\in\Omega\times(0,\infty).
\end{align*}
Finally, then the Gronwall Lemma yields the result.
\end{proof}

With that tools we can derive derive local-in-time $L^p(\Omega)$-bounds for $u_\eps$.

\section{Local-in-time but uniform-in-$\eps$ boundedness of $u_\eps$}
Our aim is to find (local-in-time) $L^\infty$-bounds for $u_\eps$ that are independent of $\eps$.
We start with the derivation of $L^p$-bounds
for small $p$. To derive this we recall an useful estimate from \cite{FujieSenba2016}.
\begin{lem}\label{Lp+1 Abschätzung für p in 1,2}
Suppose \eqref{Generalvoraussetzungen} and let $p\in(1,2)$. Then there is $C(p)>0$ such that for all
$t\in[0,\infty)$, $s>1$ and $\eps\in(0,1)$ the corresponding classical solution $(u_\eps,v_\eps)$ satisfies
 \begin{align*}
  \io u_\eps^{p+1}(\cdot,t)\leq&\frac{C(p+1)^2}{\ln(s)}\io\left( u_\eps\ln(u_\eps)+\frac{1}{e} \right)
  \cdot\io u_\eps^{p-2}\vert\nabla u_\eps\vert^2
  +6s^{p+1}\vert\Omega\vert+4C^2\vert\Omega\vert^{2-p}\Vert u_0\Vert_{L^1(\Omega)}^{p+1}.
 \end{align*}
\end{lem}
\begin{proof}
 \cite[Lemma 2.4 ii)]{FujieSenba2016}.
\end{proof}
By a testing procedure, we derive an ODI for the $L^p$-norms of $u_\eps$.
\begin{lem}\label{ODI Lp Norm}
Suppose \eqref{Generalvoraussetzungen} and let $T>0$. Moreover, let $(u_\eps,v_\eps)$ be a classical solution
of \eqref{Hauptsystem}. Then there is $C(T)>0$ such that for all $p\geq 1$ and $\eps,\eta\in(0,1)$,
we have
\begin{align*}
 \frac{1}{p}\frac{d}{dt}\io u_\eps^p\leq- C(T)(1-\eta)(p-1)\io u_\eps^{p-2}\vert\nabla u_\eps\vert^2
 +\frac{p-1}{4\eta}\io u_\eps^p\vert\nabla v_\eps\vert^2
 +\frac{\kappa^{p+1}}{\mu^p}\vert\Omega\vert.
\end{align*}
\end{lem}
\begin{proof}
We test the parabolic equation in \eqref{Hauptsystem} by $u_\eps^{p-1}$ and integrate by parts, which yield
\begin{align*}
 \frac{1}{p}\frac{d}{dt}\io u_\eps^p=&-(p-1)\io u_\eps^{p-2}\nabla u_\eps\cdot\nabla(u_\eps e^{-v_\eps})
 +\eps\kappa\io u_\eps^p-\eps\mu\io u_\eps^{p+1}\\
 =&-(p-1)\io e^{-v_\eps}u_\eps^{p-2}\vert\nabla u_\eps\vert^2
 +(p-1)\io e^{-v_\eps}u_\eps^{p-1}\nabla u_\eps\cdot\nabla v_\eps
 +\eps\kappa\io u_\eps^p-\eps\mu\io u_\eps^{p+1}.
\end{align*}
For the second term we apply Young's inequality for some $\eta\in(0,1)$,
while we use Lemma \ref{Lemma Abschätzung Logistik} for the logistic term to obtain
\begin{align*}
 \frac{1}{p}\frac{d}{dt}\io u_\eps^p\leq-(1-\eta)(p-1)\io e^{-v_\eps}u_\eps^{p-2}\vert\nabla u_\eps\vert^2
 +\frac{p-1}{4\eta}\io e^{-v_\eps}u_\eps^p\vert\nabla v_\eps\vert^2
 +\frac{\kappa^{p+1}}{\mu^p}\vert\Omega\vert.
\end{align*}
Since $v_\eps\geq 0$ by Lemma \ref{lokale Existenz} in $\Omega\times(0,T)$,
it holds $e^{-v_\eps}\leq 1$ for all $\eps\in(0,1)$. Due to the local-in-time bound for $v_\eps$ from
Lemma \ref{Zeitabhängige obere Schranke für v} there is
$C(T)>0$ such that $e^{-v_\eps}\geq C(T)$ in $\Omega\times(0,T)$ for all $\eps\in(0,1)$. This implies
\begin{align*}
 \frac{1}{p}\frac{d}{dt}\io u_\eps^p\leq -C(T)(1-\eta)(p-1)\io u_\eps^{p-2}\vert\nabla u_\eps\vert^2+
 \frac{p-1}{4\eta}\io u_\eps^p\vert\nabla v_\eps\vert^2+\frac{\kappa^{p+1}}{\mu^p}\vert\Omega\vert.
\end{align*}

\end{proof}

We will prepare the derivation of a $L^p$-bound of $u_\eps$ for $p\in(1,2)$. If we take a look at the estimate in Lemma
\ref{Lp+1 Abschätzung für p in 1,2}, then it is necessary to control the term $\io u_\eps\ln(u_\eps)$.
Therefore, we show that $u_\eps$ is bounded in $L^2(\Omega\times(0,T))$.
The term $\io u_\eps^{p-2}\vert\nabla u_\eps\vert^2$ is a good term,
since we obtain it after integrating by parts in the first
equation after testing it by $u_{\eps}^{p-1}$.
\begin{lem}\label{Raum-Zeit-Beschränktheit von ueps in L2}
Suppose \eqref{Generalvoraussetzungen} and let $T\in(0,\infty)$.
Then there is $C(T)>0$ such that for all $\eps\in(0,1)$
the corresponding classical solution
$(u_\eps,v_\eps)$ of \eqref{Hauptsystem} satisfies
 \begin{align*}
  \int_0^t\io u_{\eps}^2\leq C(T)\quad\text{for all } t\in(0,T).
 \end{align*}
\end{lem}
\begin{proof}
 Testing the first equation of \eqref{Hauptsystem} with $v_{\eps}$
 and using the second equation for $u_{\eps t}$ yields
 \begin{align*}
  \io u_{\eps t}v_{\eps}=\io(-\Delta v_{\eps}+v_{\eps})_t v_{\eps}
  =\io\nabla v_{\eps t}\cdot\nabla v_{\eps}+\io v_{\eps t}v_{\eps}=\frac{1}{2}
  \frac{d}{dt}\io(v_{\eps}^2+\vert\nabla v_{\eps}\vert^2),
 \end{align*}
 in $(0,T)$ for each $\eps\in(0,1)$.
Integrating by parts twice and using $\Delta v_{\eps}=v_{\eps}-u_{\eps}$ then yields
\begin{align*}
 &\frac{1}{2}\frac{d}{dt}\io(v_{\eps}^2+\vert\nabla v_{\eps}\vert^2)=\io\Delta(u_{\eps} e^{-v_{\eps}})v_{\eps}
 +\eps\kappa\io u_{\eps} v_{\eps}-\eps\mu\io u_{\eps}^2v_{\eps}\\
 =&\io u_{\eps} e^{-v_{\eps}}\Delta v_{\eps}+\eps\kappa\io u_{\eps} v_{\eps}
 -\eps\mu\io u_{\eps}^2 v_{\eps}\\
 =&\io u_{\eps} e^{-v_{\eps}}(v_{\eps}-u_{\eps})
 +\eps\kappa\io u_{\eps} v_{\eps}-\eps\mu\io u_{\eps}^2v_{\eps}\\
 =&\io u_{\eps} v_{\eps} e^{-v_{\eps}}-\io e^{-v_{\eps}}u_{\eps}^2
 +\eps\kappa\io u_{\eps} v_{\eps}-\eps\mu\io u_{\eps}^2v_{\eps}\quad\text{in }(0,T),\:
 \text{for each }\eps\in(0,1).
\end{align*}
Since $v_{\eps}$ is nonnegative, it holds that $v_{\eps}(x,t)e^{-v_{\eps}(x,t)}\leq e^{-1}$
for all $(x,t)\in\Omega\times(0,T)$ and $\eps\in(0,1)$.
In combination with Lemma \ref{Beschränktheit von u in L^1}, we obtain
\begin{align*}
 \io u_{\eps} v_{\eps} e^{-v_{\eps}}\leq\frac{1}{e}\io u_{\eps}\leq\frac{m}{e}\quad\text{in }(0,T),
 \:\text{for each }\eps\in(0,1).
\end{align*}
By Young's inequality and Lemma \ref{u=v in L^1}, we have
\begin{align*}
 &\eps\kappa\io u_{\eps} v_{\eps}-\eps\mu\io u_{\eps}^2v_{\eps}
 \leq\frac{\mu}{\kappa}\eps\kappa\io u_{\eps}^2v_{\eps}
 +\frac{\kappa^2\eps}{4\mu}\io v_{\eps}
 -\eps\mu\io u_{\eps}^2 v_{\eps}\\
 =&\frac{\kappa^2\eps}{4\mu}\io v_{\eps}=\frac{\kappa^2\eps}{4\mu}\io u_{\eps}\leq\frac{\kappa^2\eps m}{4\mu}
 \leq \frac{\kappa^2 m}{4\mu},
\end{align*}
in $(0,T)$ for each $\eps\in(0,1)$. Combining these estimates yields
\begin{align*}
 \frac{1}{2}\Vert v_{\eps}(\cdot,t)\Vert_{W^{1,2}(\Omega)}^2+\int_0^t\io e^{-v_{\eps}}u_{\eps}^2
 \leq\Vert v_0\Vert_{W^{1,2}(\Omega)}^2+
 \left( \frac{m}{e}+\frac{\kappa^2m}{4\mu} \right)T,\quad\text{for all }t\in(0,T),\:\eps\in(0,1).
\end{align*}
Since $u_0\in L^2(\Omega)$, we find constants $C_1,C_2>0$ by Lemma \ref{elliptische Regularitätstheorie} such that
\begin{align*}
 \Vert v_0\Vert_{W^{1,2}(\Omega)}^2\leq C_1\Vert u_0\Vert_{L^2(\Omega)}^2\leq C_2.
\end{align*}
Since by Lemma \ref{Zeitabhängige obere Schranke für v},
$v_{\eps}$ is locally bounded in time, we find $C_3(T)>0$ such that
\begin{align*}
 \frac{1}{2}\Vert v_{\eps}(\cdot,t)\Vert_{W^{1,2}(\Omega)}^2+C_3(T)\int_0^t\io u_{\eps}^2
 \leq C_2+\left( \frac{m}{e}
 +\frac{\kappa^2m}{4\mu} \right),\quad\text{for all }t\in(0,T),\:\eps\in(0,1),
\end{align*}
which proves the statement.
\end{proof}

With that tool, we can show boundedness of $\io u_\eps(\cdot,t)\ln(u_\eps(\cdot,t))$ for all $t<T$.

\begin{lem}\label{uln(u) beschränkt in L^1}
Suppose \eqref{Generalvoraussetzungen} and let $T\in(0,\infty)$.
Then there is $C(T)>0$ such that for all $\eps\in(0,1)$ the corresponding
classical solution
$(u_\eps,v_\eps)$ of \eqref{Hauptsystem} satisfies
 \begin{align*}
  \sup_{t\in(0,T)}\io u_{\eps}(\cdot,t)\ln(u_{\eps}(\cdot,t))\leq C(T).
 \end{align*}
\end{lem}
\begin{proof}
 We test the first equation by $\ln(u_\eps)$ and compute
 \begin{align*}
  &\frac{d}{dt}\io u_{\eps}\ln(u_{\eps})=\io u_{\eps t}\ln(u_{\eps})+\io u_{\eps t}\\
  =&-\io e^{-v_{\eps}}\frac{\vert\nabla u_{\eps}\vert^2}{u_{\eps}}
  +\io e^{-v_{\eps}}\nabla u_{\eps}\cdot\nabla v_{\eps}
  +\eps\io(\kappa u_{\eps}\ln(u_{\eps})-\mu u_{\eps}^2\ln(u_{\eps}))
  +\eps\io(\kappa u_{\eps}-\mu u_{\eps}^2),
 \end{align*}
 in $(0,T)$ for each $\eps\in(0,1)$.
Now we can apply Lemma \ref{Lemma Abschätzung Logistik} for $p=1$ to find $C_1>0$ such that
\begin{align*}
 &\frac{d}{dt}\io u_{\eps}\ln(u_{\eps})\leq-\io e^{-v_{\eps}}\frac{\vert\nabla u_{\eps}\vert^2}{u_{\eps}}
 +\io e^{-v_{\eps}}\nabla u_{\eps}\cdot\nabla v_{\eps}
 +\eps C_1\vert\Omega\vert+\eps\frac{\kappa^2}{\mu}\vert\Omega\vert\\
 \leq&-\io e^{-v_{\eps}}\frac{\vert\nabla u_{\eps}\vert^2}{u_{\eps}}
 +\io e^{-v_{\eps}}\nabla u_{\eps}\cdot\nabla v_{\eps}
 + C_1\vert\Omega\vert+\frac{\kappa^2}{\mu}\vert\Omega\vert\quad\text{in }(0,T)\:\text{for each }\eps\in(0,1).
\end{align*}
We apply Young's inequality twice to estimate the second integral
\begin{align*}
 &\io e^{-v_{\eps}}\nabla u_{\eps}\cdot\nabla v_{\eps}=\io e^{-v_{\eps}}
 \frac{u_{\eps}}{u_{\eps}}\nabla u_{\eps}\cdot\nabla v_{\eps}
 \leq\frac{1}{2}\io e^{-v_{\eps}}\frac{\vert\nabla u_{\eps}\vert^2}{u_{\eps}}
 +\frac{1}{2}\io e^{-v_{\eps}}u_{\eps}\vert\nabla v_{\eps}\vert^2\\
 \leq&\frac{1}{2}\io e^{-v_{\eps}}\frac{\vert\nabla u_{\eps}\vert^2}{u_{\eps}}
 +\frac{1}{4}\io e^{-v_{\eps}}u_{\eps}^2
 +\frac{1}{4}\io e^{-v_{\eps}}
 \vert\nabla v_{\eps}\vert^4\quad\text{in }(0,T)\:\text{for each }\eps\in(0,1).
\end{align*}
By Gagliardo-Nirenberg's inequality and the estimate
\begin{align*}
 \Vert D^2 v_{\eps}\Vert_{L^2(\Omega)}\leq\Vert v_{\eps}\Vert_{W^{2,2}(\Omega)}
\end{align*}
we find $C_2,C_3>0$ such that
\begin{align*}
 \Vert\nabla v_{\eps}\Vert_{L^4(\Omega)}
 \leq C_2\Vert v_{\eps}\Vert_{W^{2,2}}^{1/2}\Vert v_{\eps}\Vert_{L^\infty(\Omega)}^{1/2}+C_3
 \Vert v_{\eps}\Vert_{L^\infty(\Omega)}\quad\text{in }(0,T)\:\text{for each }\eps\in(0,1).
\end{align*}
Then Young's inequality and Lemma \ref{Zeitabhängige obere Schranke für v} imply the existence of $C_4(T),C_5(T)>0$
such that
\begin{align*}
 \Vert\nabla v_{\eps}\Vert_{L^4(\Omega)}
 \leq 8C_2^4\Vert v_{\eps}\Vert_{W^{2,2}(\Omega)}^2\Vert v_{\eps}\Vert_{L^\infty(\Omega)}^2
 +8C_3^4\Vert v_{\eps}\Vert_{L^\infty(\Omega)}^4
 \leq C_4(T)\Vert v_{\eps}\Vert_{W^{2,2}(\Omega)}^2+C_5(T),
\end{align*}
in $(0,T)$ for each $\eps\in(0,1)$.
Lemma \ref{elliptische Regularitätstheorie} implies the existence
of a constant $C_6(T)>0$
such that
\begin{align*}
 C_4(T)\Vert v_{\eps}\Vert_{W^{2,2}(\Omega)}^2+C_5(T)\leq C_6(T)\Vert u_{\eps}\Vert_{L^2(\Omega)}^2+C_5(T),
\end{align*}
in $(0,T)$ for each $\eps\in(0,1)$.
Hence,
\begin{align*}
 \io u_{\eps}(\cdot,t)\ln(u_{\eps}(\cdot,t))\leq&\io u_0\ln(u_0)+\frac{1}{4}\int_0^t\io e^{-v_{\eps}}u_{\eps}^2
 +\frac{C_6(T)}{4}\int_0^t
 \io u_{\eps}^2+C_7(T)\\
 \leq&C_8+\frac{C_6(T)+1}{4}\int_0^t\io u_{\eps}^2+C_7(T)\quad\text{for all }t\in(0,T),\:\eps\in(0,1),
\end{align*}
where $C_7(T):=\frac{C_6(T)\cdot C_5(T)}{4}+C_1\vert\Omega\vert+\frac{\kappa^2}{\mu}\vert\Omega\vert$.
The constant $C_8>0$ can be obtained due to $u_0\in L^\infty(\Omega)$.
Due to Lemma \ref{Raum-Zeit-Beschränktheit von ueps in L2},
\begin{align*}
 \io u_{\eps}(\cdot,t)\ln(u_{\eps}(\cdot,t))\leq C_9(T),
\end{align*}
for all $t\in(0,T)$ and $\eps\in(0,1)$, which proves the statement.
\end{proof}

With the previous Lemma we can now show that $u_\eps$ has a local-in-time bound
in $L^p(\Omega)$ for $p\in(1,2)$.

\begin{lem}\label{Beschränktheit von ueps in Lp für p in 1,2}
Suppose \eqref{Generalvoraussetzungen}.
 Let $p\in(1,2)$ and $T\in(0,\infty)$. Then there is $C(p,T)>0$ such that for all $\eps\in(0,1)$
 the corresponding classical solution $(u_\eps,v_\eps)$ of \eqref{Hauptsystem} satisfies
 \begin{align*}
  \Vert u_{\eps}(\cdot,t)\Vert_{L^p(\Omega)}\leq C(p,T)\quad\text{for all } t\in(0,T).
 \end{align*}
\end{lem}
\begin{proof}
 We test the first equation in \eqref{Hauptsystem} with $u_{\eps}^{p-1}$,
 integrate over $\Omega$, integrate by parts and apply Lemma \ref{ODI Lp Norm} with $\eta=\frac{1}{2}$
 to obtain a constant $C_1(T)>0$ such that
 \begin{align*}
  &\frac{1}{p}\frac{d}{dt}\io u_{\eps}^p\leq
  -\frac{C_1(T)(p-1)}{2}\io u_\eps^{p-2}\vert\nabla u_\eps\vert^2
  +\frac{p-1}{2}\io u_\eps^p\vert\nabla v_\eps\vert^2+\frac{\kappa^{p+1}}{\mu^p}\vert\Omega\vert,
 \end{align*}
 in $(0,T)$ for each $\eps\in(0,1)$.
By Hölder's and Young's inequality we obtain
\begin{align*}
 \io u_{\eps}^{p+1}\vert\nabla v_{\eps}\vert^2\leq\left( \io u_{\eps}^{p+1} \right)^{\frac{p}{p+1}}
 \left( \io \vert\nabla v_{\eps}\vert^{2(p+1)} \right)^\frac{1}{p+1}
 \leq\frac{p}{p+1}\io u_{\eps}^{p+1}+\frac{1}{p+1}\io \vert\nabla v_{\eps}\vert^{2(p+1)},
\end{align*}
in $(0,T)$ for each $\eps\in(0,1)$.
The Sobolev inequality and Lemma \ref{elliptische Regularitätstheorie}
yield the existence of $C_2,C_3>0$ such that
\begin{align*}
 \Vert\nabla v_{\eps}\Vert_{L^{2(p+1)}(\Omega)}\leq C_2\Vert v_{\eps}\Vert_{W^{2,\frac{2(p+1)}{p+2}}(\Omega)}\leq
 C_3\Vert u_{\eps}\Vert_{L^\frac{2(p+1)}{p+2}(\Omega)}\quad\text{in }(0,T)\:\text{for each }\eps\in(0,1).
\end{align*}
An interpolation estimate \cite[Theorem 4.6, Remark 2]{brezis2011functional}
and Lemma \ref{Beschränktheit von u in L^1} imply
\begin{align*}
 C_3\Vert u_{\eps}\Vert_{L^\frac{2(p+1)}{p+2}(\Omega)}^{2(p+1)}
 \leq C_3\Vert u_{\eps}\Vert_{L^1(\Omega)}^{p+1}\io u_{\eps}^{p+1}\leq
 C_3m^{p+1}\io u_{\eps}^{p+1}\quad\text{in }(0,T)\:\text{for each }\eps\in(0,1).
\end{align*}
 By Lemma \ref{Lp+1 Abschätzung für p in 1,2} there is $C_4>0$ such that
\begin{align*}
 \frac{p+C_3m^{p+1}}{p+1}\io u_{\eps}^{p+1}\leq&\frac{p+C_3m^{p+1}}{p+1}\cdot(\frac{C_4(p+1)^2}{\ln(s)}
 \io\left( u_{\eps}\ln(u_{\eps})+\frac{1}{e} \right)\io u_{\eps}^{p-2}\vert\nabla u_{\eps}\vert^2\\
 &+6s^{p+1}\vert\Omega\vert+4C_3^2\vert\Omega\vert^{2-p}\Vert u_0\Vert_{L^1(\Omega)}^{p+1})
 \quad\text{in }(0,T)\:\text{for each }\eps\in(0,1).
\end{align*}
By Lemma \ref{uln(u) beschränkt in L^1} there is $C_5(T)>0$ such that for all $s>1$
\begin{align*}
 \frac{p+C_2m^{p+1}}{p+1}\io u_{\eps}^{p+1}\leq&\frac{p+C_3m^{p+1}}{p+1}\cdot(\frac{C_4(p+1)^2C_5(T)}{\ln(s)}
 \io u_{\eps}^{p-2}\vert\nabla u_{\eps}\vert^2\\
 &+6s^{p+1}\vert\Omega\vert+4C_3^2\vert\Omega\vert^{2-p}\Vert u_0\Vert_{L^1(\Omega)}^{p+1})
 \quad\text{in }(0,T)\:\text{for each }\eps\in(0,1).
\end{align*}
We choose $s$ large enough such that
\begin{align*}
 C_6(T):=-\frac{C_1(T)(p-1)}{2}+\frac{p+C_3m^{p+1}}{p+1}\cdot\frac{C_4(p+1)^2C_5(T)}{\ln(s)}\leq 0.
\end{align*}
This implies the existence of $C_7>0$ such that
\begin{align*}
 \frac{1}{p}\frac{d}{dt}\io u_{\eps}^p\leq C_7+\frac{\kappa^2}{\mu}\vert\Omega\vert\quad\text{in }
 (0,T)\:\text{for each }\eps\in(0,1),
\end{align*}
which proves the statement.
\end{proof}
To derive a $L^\infty$-bound for $u_\eps$ we will perform a Moser iteration.
For that our aim is to find a local-in-time bound
for $\Vert\nabla v_\eps\Vert_{L^\infty(\Omega)}$. Due to the embedding
$W^{2,q}(\Omega)\hookrightarrow W^{1,\infty}(\Omega)$ since $q>2$
and Lemma \ref{elliptische Regularitätstheorie},
it is sufficient to show that $u_\eps$ is local-in-time bounded in $L^3(\Omega)$.

\begin{lem}\label{L3 Schranke für u}
Suppose \eqref{Generalvoraussetzungen} and let $T\in(0,\infty)$. Then there is
$C(T)>0$ such that for all $\eps\in(0,1)$ the corresponding classical solution $(u_\eps,v_\eps)$ of
 \eqref{Hauptsystem} satisfies
 \begin{align*}
  \sup_{t\in(0,T)}\Vert u_{\eps}(\cdot,t)\Vert_{L^3(\Omega)}\leq C(T).
 \end{align*}
\end{lem}
\begin{proof}
 We the parabolic equation in \eqref{Hauptsystem} by $u_{\eps}^2$ and apply Lemma \ref{ODI Lp Norm}
 with $\eta_1=\frac{1}{2}$ to obtain a constant $C_1(T)>0$ such that
 \begin{align*}
  \frac{1}{3}\frac{d}{dt}\io u_\eps^3\leq-C_1(T)\io u_\eps\vert\nabla u_\eps\vert^2
  +\io u_\eps^3\vert\nabla v_\eps\vert^2+\frac{\kappa^4}{\mu^3}\vert\Omega\vert.
 \end{align*}
We use the identity
$u_\eps\vert\nabla u_\eps\vert^2=\frac{4}{9}\vert\nabla u_\eps^\frac{3}{2}\vert^2$ and apply Young's
inequality to the second integral, which yield
 \begin{align*}
  \frac{1}{3}\frac{d}{dt}\io u_{\eps}^3\leq
  -\frac{4C_1(T)}{9}\io\vert\nabla u_{\eps}^\frac{3}{2}\vert^2
  +\frac{3}{4}\io u_{\eps}^4+\frac{1}{4}\io\vert\nabla v_{\eps}\vert^8+\frac{\kappa^4}{\mu^3}\vert\Omega\vert
  \quad\text{in }(0,T)\:\text{for each }\eps\in(0,1).
 \end{align*}
By the Gagliardo-Nirenberg inequality
and Lemma \ref{Zeitabhängige obere Schranke für v} we find constants $C_2(T),C_3(T)>0$ such that
\begin{align*}
 \Vert\nabla v_{\eps}\Vert_{L^8(\Omega)}^8\leq
 C_2(T)\Vert u_{\eps}\Vert_{L^4(\Omega)}^4+C_3(T)\quad\text{in }(0,T)\:\text{for each }\eps\in(0,1).
\end{align*}
By the Gagliardo-Nirenberg inequality there are $C_4,C_5>0$ and by Lemma
\ref{Beschränktheit von ueps in Lp für p in 1,2} we find constants $C_6(T),C_7(T)>0$ such that
\begin{align*}
 \io u_{\eps}^4=\Vert u_{\eps}^\frac{3}{2}(\cdot,t)\Vert_{L^\frac{8}{3}(\Omega)}^\frac{8}{3}
 \leq&\left( C_4\Vert u_{\eps}^\frac{3}{2}\Vert_{L^1(\Omega)}^\frac{3}{8}
 \Vert\nabla u_{\eps}^\frac{3}{2}\Vert_{L^2(\Omega)}^\frac{5}{8}
 +C_5\Vert u_{\eps}^\frac{3}{2}\Vert_{L^1(\Omega)} \right)^\frac{8}{3}\\
 \leq&C_6(T)\Vert\nabla u_{\eps}^\frac{5}{3}\Vert_{L^2(\Omega)}+C_7(T)\\
 \leq&\eta_2\Vert\nabla u_{\eps}^\frac{3}{2}\Vert_{L^2(\Omega)}^2
 +\frac{1}{6}\cdot\left( \frac{6\eta_2}{5} \right)^{-5}
 \cdot(C_6(T))^6+C_7(T)\quad\text{in }(0,T)\:\text{for each }\eps\in(0,1),
\end{align*}
where we set $\eta_2=\frac{4}{C_2(T)+3}\cdot\frac{4C_1(T)}{9}$ to find $C_8(T)>0$ such that
\begin{align*}
 \frac{1}{3}\frac{d}{dt}\io u_{\eps}^3(\cdot,t)\leq C_8(T)\quad\text{for all }t\in(0,T),\:\eps\in(0,1).
\end{align*}
\end{proof}

Now we can do a Moser iteration to show that $u_\eps$ is locally-in-time bounded.
We follow the strategy and notation from \cite[Lemma 2.11]{FuestLankeitMizukami2021}.

\begin{lem}\label{Linfty Schranke für u}
Suppose \eqref{Generalvoraussetzungen} and let $T\in(0,\infty)$. Moreover, let $L>0$ such that
 $\Vert u_0\Vert_{L^\infty(\Omega)}\leq L$.
 Then there is $C(T)>0$ such that for all $\eps\in(0,1)$ the corresponding classical solution
 $(u_\eps,v_\eps)$ of \eqref{Hauptsystem} satisfies
 \begin{align*}
  \sup_{t\in(0,T)}\Vert u_\eps(\cdot,t)\Vert_{L^\infty(\Omega)}\leq C(T).
 \end{align*}
\end{lem}

\begin{proof}
  We test the parabolic equation of \eqref{Hauptsystem} by $u_\eps^{p-1}$,
  and apply Lemma \ref{ODI Lp Norm} with $\eta_1=\frac{1}{2}$ to find a constant $C_1(T)>0$ such that
    \begin{align*}
        \frac{1}{p}\frac{d}{dt}\io u_\eps^p
        \leq-\frac{C_1(T)(p-1)}{2}\io u_\eps^{p-2}\vert\nabla u_\eps\vert^2
        +\frac{p-1}{2}\io u_\eps^p\vert\nabla v_\eps\vert^2
        +\kappa\left(\frac{\kappa}{\mu}\right)^p\vert\Omega\vert,
    \end{align*}
    in $(0,T)$ for each $\eps\in(0,1)$.
    Hölder's inequality, Lemma \ref{elliptische Regularitätstheorie}
    and Lemma \ref{L3 Schranke für u} imply the existence of $C_2,C_3,C_4(T)>0$ such that
    \begin{align*}
    \begin{split}
        \io u_\eps^p\vert\nabla v_\eps\vert^2
        \leq\Vert\nabla v_\eps\Vert_{L^\infty(\Omega)}^2\io u_\eps^p
        \leq C_2\Vert v_\eps\Vert_{W^{2,3}(\Omega)}^2\io u_\eps^p
        \leq C_3\Vert u_\eps\Vert_{L^3(\Omega)}^2\io u_\eps^p
        \leq C_4(T)\io u_\eps^p,
        \end{split}
    \end{align*}
    in $(0,T)$ for each $\eps\in(0,1)$.
    Due to the Gagliardo-Nirenberg inequality, we find $C_5,C_6>0$
    and apply Young's inequality with $\eta_2>0$ such that
    \begin{align*}
        \io u_\eps^p=\Vert u_\eps^\frac{p}{2}\Vert_{L^2(\Omega)}^2
        &\leq(C_5\Vert\nabla u_\eps^\frac{p}{2}\Vert_{L^2(\Omega)}^\frac{1}{2}
        \Vert u_\eps^\frac{p}{2}\Vert_{L^1(\Omega)}^\frac{1}{2}
        +C_6\Vert u_\eps^\frac{p}{2}\Vert_{L^1(\Omega)})^2\\
        &\leq2C_5^2\Vert\nabla u_\eps^\frac{p}{2}\Vert_{L^2(\Omega)}\Vert u_\eps^\frac{p}{2}
        \Vert_{L^1(\Omega)}+2C_6^2\Vert u_\eps^\frac{p}{2}\Vert_{L^1(\Omega)}^2\\
        &\leq\eta_2\Vert\nabla u_\eps^\frac{p}{2}\Vert_{L^2(\Omega)}^2
        +\frac{1}{\eta_2}C_5^4\Vert u_\eps^\frac{p}{2}\Vert_{L^1(\Omega)}^2
        +2C_6^2\Vert u_\eps^\frac{p}{2}\Vert_{L^1(\Omega)}^2\quad\text{in }(0,T)\:\text{for each }\eps\in(0,1).
    \end{align*}
    Since $p-1\geq\frac{p}{2}\iff p\geq 2$, we find $C_7>0$ such that
   \begin{align*}
       -\frac{2C_1(T)(p-1)}{p^2}\leq-\frac{2C_1(T)\frac{p}{2}}{C_7p^2}=-\frac{C_1(T)}{C_7p},
   \end{align*}
   for all $p\geq 2$. If we apply the obvious inequality $\frac{p-1}{2}\leq\frac{p}{2}$
   for all $p\geq 1$ and choose $\eta_2=\frac{2C_1(T)}{C_4(T)C_7p^2}$, then we have due to the
   identity $u_\eps^{p-2}\vert\nabla u_\eps\vert^2=\frac{4}{p^2}\vert\nabla u_\eps^\frac{p}{2}\vert^2$
   \begin{align}\label{Moser-Iteration, 1/p d/dt}
       \frac{1}{p}\frac{d}{dt}\io u^p_\eps(\cdot,t)\leq\frac{C_4^2(T)C_5^4C_7p^2}{2C_1(T)}
       \left(\io u_\eps^\frac{p}{2}(\cdot,t)\right)^2
       +C_4(T)C_6^2p\left(\io u_\eps^\frac{p}{2}(\cdot,t)\right)^2
       +\kappa\left(\frac{\kappa}{\mu}\right)^p\vert\Omega\vert,
   \end{align}
   for all $t\in(0,T)$ and $\eps\in(0,1)$.
   Now we multiply \eqref{Moser-Iteration, 1/p d/dt} by $p$ and
   integrate over $(0,T')$ for arbitrary $T'\in(0,T)$, which yields
   \begin{align*}
       &\io u_\eps^p(\cdot,T')\leq\io u_0^p+\left(\frac{C_4^2(T)C_5^4C_7p^3}{2C_1(T)}+C_4(T)C_6^2p^2\right)\int_0^{T'}
       \left(\io u_\eps^\frac{p}{2}(\cdot,s)ds\right)^2+\kappa pT\vert\Omega\vert\left(\frac{\kappa}{\mu}\right)^p\\
       \leq&\io u_0^p+T\left(\frac{C_4^2(T)C_5^4C_7p^3}{2C_1(T)}+C_4(T)C_6^2p^2\right)
       \sup_{t\in(0,T)}\left(\io u_\eps^\frac{p}{2}(\cdot,t)\right)^2
       +\kappa pT\vert\Omega\vert\left(\frac{\kappa}{\mu}\right)^p\\
       \leq&\io u_0^p+Tp^3\left(\frac{C_4^2(T)C_5^4C_7}{2C_1(T)}+C_4(T)C_6^2\right)
       \sup_{t\in(0,T)}\left(\io u_\eps^\frac{p}{2}(\cdot,t)\right)^2+\kappa pT\vert\Omega\vert\left(\frac{\kappa}{\mu}\right)^p,
   \end{align*}
   for all $\eps\in(0,1)$, where we used $p^2\leq p^3$ for all $p\geq 1$. Hence,
   \begin{align*}
       \io u_\eps^p(\cdot,T')\leq\max\left\{2\left(\io u_0^p+\kappa pT\vert\Omega\vert\left(\frac{\kappa}{\mu} \right)^p \right),
       TC_8(T)p^3\sup_{t\in(0,T')}\left(\io u_\eps^\frac{p}{2}(\cdot,t) \right)^2 \right\},
   \end{align*}
   for all $\eps\in(0,1)$, $p\in\mathbb{N}\setminus\{1\}$
   and where $C_8(T):=2\left(\frac{C_4^2(T)C_5^4C_7}{2C_1(T)}+C_4(T)C_6^2\right)$.
   Since $T'$ is arbitrary, the estimate holds true for all $T'\in[0,T]$.
   Let us define a sequence $(A_{j,\eps})_{j\in\mathbb{N}}\subseteq(0,\infty)$ as
   $A_{j,\eps}:=\sup_{t\in(0,T)}\Vert u_\eps(\cdot,t)\Vert_{L^{p_j}(\Omega)}$,
   where $(p_j)_{j\in\mathbb{N}}\subseteq[1,\infty)$ with $p_j:=2^j$. With that choice we have
   \begin{align*}
       A_{j,\eps}&\leq\Vert u_0\Vert_{L^{p_j}(\Omega)}+\left(TC_8(T)p_j^3\right)^\frac{1}{p_j}A_{j-1,\eps}
       +\left(\kappa p_jT\vert\Omega\vert\right)^\frac{1}{p_j}\left(\frac{\kappa}{\mu}\right)\\
       &\leq\max\{2^{\frac{1}{p_j}}\Vert u_0\Vert_{L^{p_j}(\Omega)}
       +2^{\frac{1}{p_j}}\left(\kappa p_jT\vert\Omega\vert\right)^\frac{1}{p_j}\left(\frac{\kappa}{\mu}\right),
       \left(TC_8(T)p_j^3\right)^\frac{1}{p_j}A_{j-1,\eps}\}.
   \end{align*}
   Now we introduce the sequence $(B_{j,\eps})_{j\in\mathbb{N}}\subseteq(0,\infty)$ such that
   \begin{align*}
       B_{0,\eps}=L\cdot\max\{1,\vert\Omega\vert\}+\frac{T\kappa\vert\Omega\vert}{\mu},
       \quad B_{j,\eps}=\max\{2L\cdot\max\{1,\vert\Omega\vert\}
       +2C_9(T)\frac{\kappa}{\mu},(TC_8(T)p_j^3)^\frac{1}{p_j}B_{j-1,\eps}\},
   \end{align*}
   for all $j\in\mathbb{N}$, where $C_9(T):=\sup_{l\in\mathbb{N}}
   \{(\kappa p_lT\vert\Omega\vert)^\frac{1}{p_l}\}<\infty$.
   Boundedness of $C_9(T)$ follows by $\lim_{j\to\infty}(\kappa p_jT\vert\Omega\vert)^\frac{1}{p_j}=1$,
   because of convergence.
   By Lemma \ref{Beschränktheit von u in L^1} we have $A_{0,\eps}\leq B_{0,\eps}$ for all $\eps\in(0,1)$.
   In addition it holds
   \begin{align*}
       2^\frac{1}{p_j}\Vert u_0\Vert_{L^{p_j}(\Omega)}
       +2^\frac{1}{p_j}(\kappa p_jT\vert\Omega\vert)^\frac{1}{p_j}
       \frac{\kappa}{\mu}\leq 2\Vert u_0\Vert_{L^\infty(\Omega)}\cdot\max\{1,\vert\Omega\vert\}
       +2\frac{\kappa}{\mu}C_9(T)\quad\text{for all }j\in\N.
   \end{align*}
   Due to monotonicity of the right-hand side with respect to $B_{j-1,\eps}$,
   we can conclude $A_{j,\eps}\leq B_{j,\eps}$ for all $\eps\in(0,1)$ and $j\in\mathbb{N}_0$.
   Hence, we have
   \begin{align*}
       \sup_{t\in(0,T)}\Vert u_\eps(\cdot,t)\Vert_{L^\infty(\Omega)}=
       \lim_{j\to\infty}A_{j,\eps}\leq\liminf_{j\to\infty}B_{j,\eps}\quad\text{for all }\eps\in(0,1).
   \end{align*}
   It is now sufficient to show $\liminf_{j\to\infty}B_{j,\eps}<\infty$ for all $\eps\in(0,1)$.
   If there exists a sequence $(j_k)_{k\in\mathbb{N}}\subseteq\mathbb{N}_0$ such that $j_k\to\infty$
   as $k\to\infty$ and
   $B_{j_k,\eps}\leq 2L\cdot\max\{\vert\Omega\vert,1\}+2\frac{\kappa}{\mu}C_9(T)$
   for all $k\in\mathbb{N}$, then
   \begin{align*}
       \liminf_{j\to\infty}B_{j,\eps}\leq\liminf_{k\to\infty}B_{j_k,\eps}
       \leq2L\cdot\max\{\vert\Omega\vert,1\}+2\frac{\kappa}{\mu}C_9(T).
   \end{align*}
   Now let us assume that there is $j_0\in\mathbb{N}$ such that $B_{j,\eps}>2L\cdot
   \max\{\vert\Omega\vert,1\}+2\frac{\kappa}{\mu}C_9(T)$
   for all $j\geq j_0$. This implies
   \begin{align*}
       B_{j,\eps}\leq (TC_8(T)p_j^3)^\frac{1}{p_j}B_{j-1}\quad\forall j\geq j_0.
   \end{align*}
   By induction, it follows
   \begin{align*}
       B_{j,\eps}\leq\left(\prod_{k=j_0}^j(TC_8(T)p_k^3)^\frac{1}{p_k}\right)B_{j_0,\eps}
       \leq \left((TC_8(T))^{\sum_{k=j_0}^j
       \frac{1}{2^k}}\cdot 2^{3\sum_{k=j_0}^j\frac{k}{2^k}}\right)\cdot B_{j_0,\eps},
   \end{align*}
   for all $j\geq j_0$. Due to
   \begin{align*}
       \limsup_{k\to\infty}\sqrt[k]{\frac{k}{2^k}}=\frac{1}{2},
   \end{align*}
   the series $\sum_{k=0}^\infty \frac{k}{2^k}$ converges.
   Combined with $\sum_{k=0}^\infty\frac{1}{2^k}=\frac{1}{1-\frac{1}{2}}=2$ and $B_{j_0,\eps}<\infty$,
   this yields the statement.
\end{proof}

Thanks to the independence of the local-in-time bound of $\Vert u_\eps\Vert_{L^\infty(\Omega)}$
of $\eps$ we can find bounds for solutions in Hölder spaces.

\section{Hölder continuity}
By \cite[Theorem 4]{DiBenedetto1986} we can show the existence of bounds for $u_\eps,v_\eps$ in suitable Hölder spaces. Since
$W^{1,q}(\Omega)\hookrightarrow C^\alpha(\overline{\Omega})$ for $q>2$ and some $\alpha\in(0,1)$ we obtain those constants even
for $t=0$.
\begin{lem}
Suppose \eqref{Generalvoraussetzungen} and let $T\in(0,\infty)$.
 Then there are $C_1(T),C_2(T)>0$ and $\lambda\in(0,1)$ such that
for all $\eps\in(0,1)$ the corresponding classical solution $(u_\eps,v_\eps)$ of \eqref{Hauptsystem} satisfies
 \begin{align*}
  &\Vert u_\eps\Vert_{C^{\lambda,\frac{\lambda}{2}}(\overline{\Omega}\times[0,T])}\leq C_1(T),\\
  &\Vert v_\eps\Vert_{C^{2+\lambda,\frac{\lambda}{2}}(\overline{\Omega}\times[0,T])}\leq C_2(T).
 \end{align*}
\end{lem}
\begin{proof}
 Since $u_\eps$ is locally-in-time bounded, we can apply \cite[Theorem 4]{DiBenedetto1986}. For the second
 estimate, we apply Lemma \ref{elliptische Regularitätstheorie}.
\end{proof}

Due to our $\eps$-independent bounds, we can show uniform convergence of solutions of \eqref{Hauptsystem}.
\begin{lem}\label{Gleichmäßige Konvergenz von ueps und veps}
Suppose \eqref{Generalvoraussetzungen}.
Let $(\eps_j)_{j\in\N}\subseteq(0,1)$ be a nullsequence and $T\in(0,\infty)$.
Then there is a subsequence $(\eps_{j_k})_{k\in\N}\subseteq(0,1)$ such that the corresponding classical
solutions $((u_{\eps_{j_k}},v_{\eps_{j_k}}))$ of \eqref{Hauptsystem}
satisfy
\begin{align*}
 &u_{\eps_{j_k}}\to u\quad\text{in }C^0(\overline{\Omega}\times[0,T])\quad\text{as }k\to\infty,\\
 &v_{\eps_{j_k}}\to v\quad\text{in }C^{2,0}(\overline{\Omega}\times[0,T])\quad\text{as }k\to\infty,
\end{align*}
where the limit pair satisfies
\begin{align*}
 (u,v)\in C^0(\overline{\Omega}\times[0,T])\times C^{2,0}(\overline{\Omega}\times[0,T]).
\end{align*}
\end{lem}
\begin{proof}
 The first convergence is a consequence of the Arzel\`{a}-Ascoli theorem. For the second one we use
 the computation from \cite[Proof of Lemma 4.4]{Lankeit2015}
 and obtain a constant $C>0$ by Lemma \ref{elliptische Regularitätstheorie} such that
 \begin{align*}
  \Vert v_{\eps_{j_k}}-v_{\eps_{j_l}}\Vert_{C^{2,0}(\overline{\Omega}\times[0,T])}
  \leq&\Vert v_{\eps_{j_k}}-v_{\eps_{j_l}}\Vert_{C^{2+\alpha,0}(\overline{\Omega}\times[0,T])}\\
  \leq&C\Vert u_{\eps_{j_k}}-u_{\eps_{j_l}}\Vert_{C^\alpha,0}(\overline{\Omega}\times[0,T])
  =C\Vert u_{\eps_{j_k}}-u_{\eps_{j_l}}\Vert_{C^0([0,T];\:C^\alpha(\overline{\Omega}))},
 \end{align*}
for $\alpha\in(0,\lambda)$. Hence, $(v_{\eps_{j_k}})_{k\in\N}\subseteq C^{2,0}(\overline{\Omega}\times[0,T])$
is a Cauchy sequence and thus
\begin{align*}
 v_{\eps_{j_k}}\to v\quad\text{in }C^{2,0}(\overline{\Omega}\times[0,T]).
\end{align*}
\end{proof}
We can even show that the limit functions are more regular than claimed
in Lemma \ref{Gleichmäßige Konvergenz von ueps und veps}
\begin{cor}\label{Grenzfunktionen sind Hölderstetig}
Suppose \eqref{Generalvoraussetzungen} and let $T\in(0,\infty)$. Then there
is $\alpha\in(0,\lambda)$ such that $u\in C^{\alpha,\frac{\alpha}{2}}(\overline{\Omega}\times[0,T])$,
 $v\in C^{2+\alpha,\frac{\alpha}{2}}(\overline{\Omega}\times[0,T])$,
 where $u$ and $v$ are the limit functions from Lemma \ref{Gleichmäßige Konvergenz von ueps und veps}.
\end{cor}
\begin{proof}
 \cite[Lemma 6.33]{GilbargTrudinger2001}.
\end{proof}

\section{Limit solutions}
Due to our convergence results, we can now show that the pair $(u,v)$
is the unique classical solution of \eqref{eps=0 System}. Since for now we have
$u\in C^{\alpha,\frac{\alpha}{2}}(\overline{\Omega}\times[0,T])$,
we cannot conclude that the limit pair $(u,v)$ is a classical solutions of
\eqref{eps=0 System} directly. Therefore, we introduce a weaker solution concept.
\begin{definition}\label{Stark schwache Lösungsdefinition}
Suppose \eqref{Generalvoraussetzungen} and
let $T\in(0,\infty]$. By a $W^{1,2}$-solution of \eqref{Hauptsystem} we mean a pair $(u_\eps,v_\eps)$ with
$u_\eps\in C^0(\overline{\Omega}\times[0,T))\cap L^2((0,T);\: W^{1,2}(\Omega))$
and $v_\eps\in C^{2,0}(\overline{\Omega}\times[0,T))$ such that $u_\eps,v_\eps$ are nonnegative,
$v_\eps$ satisfies
\begin{align*}
 -\Delta v_\eps+v_\eps=u_\eps\quad\text{in }\Omega,\quad \partial_\nu v_\eps=0\quad\text{on }\partial\Omega,
\end{align*}
and
\begin{align}\label{schwache Formulierung für u}
\begin{split}
 -\int_0^T\io u_\eps\varphi_t-\io u_0\varphi(\cdot,0)=&-\int_0^T\io\nabla(e^{-v_\eps}u_\eps)\cdot\nabla\varphi
 +\eps\kappa\int_0^T\io u_\eps\varphi-\eps\mu\int_0^T\io u_\eps^2\varphi
\end{split}
\end{align}
holds for all $\varphi\in C_c^\infty(\overline{\Omega}\times[0,T))$.
\end{definition}

\begin{rem}
Due to density arguments, we can reformulate identity \eqref{schwache Formulierung für u}
for $\varphi\in L^1((0,T);\:W^{1,2}(\Omega))$
that have compact support in $\overline{\Omega}\times[0,T)$
and satisfy $\varphi_t\in L^1(\Omega\times(0,T))$ and obtain the same solutions.
\end{rem}

Our aim is now to show that $(u,v)$ from Lemma \ref{Gleichmäßige Konvergenz von ueps und veps} solves \eqref{eps=0 System}
in the sense of Definition \ref{Stark schwache Lösungsdefinition}. The main challenge is to show weak convergence of
$\nabla(e^{-v_\eps}u_\eps)$ in $L^2(\Omega\times(0,T))$. Therefore, we compute some bounds in space and time for
$\nabla u_\eps$ and $\nabla(e^{-v_\eps}u_\eps)$.
\begin{lem}\label{nabla ueps in L2L2 beschränkt}
Suppose \eqref{Generalvoraussetzungen} and
let $T\in(0,\infty)$. Then there is $C(T)>0$ such that for all $\eps\in(0,1)$ the corresponding
classical solution $(u_\eps,v_\eps)$ of \eqref{Hauptsystem} satisfies
 \begin{align*}
  \int_0^T\io\vert\nabla u_{\eps}\vert^2\leq C(T).
 \end{align*}
\end{lem}
\begin{proof}
 Testing the first equation of \eqref{Hauptsystem} by $u_{\eps}$ and integrating over $\Omega$ yields
 \begin{align*}
  \frac{1}{2}\frac{d}{dt}\io u_{\eps}^2&\leq-\io e^{-v_{\eps}}\vert\nabla u_{\eps}\vert^2
  +\io e^{-v_{\eps}}u_{\eps}\vert\nabla u_{\eps}\vert\cdot
  \vert\nabla v_{\eps}\vert+\frac{\kappa^2}{\mu}\vert\Omega\vert\\
  &\leq-\frac{C_1(T)}{2}\io \vert\nabla u_{\eps}\vert^2+\frac{1}{2}\io u_{\eps}^2\vert\nabla v_{\eps}\vert^2
  +\frac{\kappa^2}{\mu}\vert\Omega\vert\quad\text{in }(0,T)\:\text{for each }\eps\in(0,1),
 \end{align*}
 due to Young's inequality and a constant $C_1(T)>0$ obtained
 from Lemma \ref{Zeitabhängige obere Schranke für v}.
 Hölder's inequality, Lemma \ref{elliptische Regularitätstheorie}
 and Lemma \ref{Linfty Schranke für u} yield the existence of $C_2,C_3(T)>0$ such that
 \begin{align*}
  \frac{1}{2}\io u_{\eps}^2\vert\nabla v_{\eps}\vert^2
  \leq\frac{1}{2}\Vert\nabla v_{\eps}\Vert_{L^\infty(\Omega)}^2\io u_{\eps}^2
  \leq\frac{C_2}{2}\Vert u_{\eps}\Vert_{L^3(\Omega)}^2\Vert u_{\eps}\Vert_{L^2(\Omega)}^2\leq C_3(T)
  \quad\text{in }(0,T)\:\text{for each }\eps\in(0,1).
 \end{align*}
 Hence,
 \begin{align*}
  \frac{1}{2}\frac{d}{dt}\io u_{\eps}^2+\frac{C_1(T)}{2}\io\vert\nabla u_{\eps}\vert^2
  \leq C_3(T)+\frac{\kappa^2}{\mu}\vert\Omega\vert\quad\text{in }(0,T)\:\text{for each }\eps\in(0,1).
 \end{align*}
Integration over $(0,T)$ yields
\begin{align*}
 \io u_{\eps}^2(\cdot,T)+C_1(T)\int_0^T\io\vert\nabla u_{\eps}\vert^2
 \leq\io u_0^2+C_3(T)T+\frac{\kappa^2T}{\mu}\vert\Omega\vert=:C_4(T)\quad\text{for all }\eps\in(0,1).
\end{align*}
\end{proof}

\begin{lem}
Suppose \eqref{Generalvoraussetzungen} and
let $T\in(0,\infty)$. Then there is $C(T)>0$ such that for all $\eps\in(0,1)$ the corresponding
classical solution
$(u_\eps,v_\eps)$ satisfies
 \begin{align*}
  \int_0^T\io\vert\nabla(u_{\eps} e^{-v_{\eps}})\vert^2\leq C(T).
 \end{align*}
\end{lem}

\begin{proof}
A direct computation yields
\begin{align*}
 &\int_0^T\io\vert\nabla u_{\eps} e^{-v_{\eps}}\vert^2
 =\int_0^T\io\vert e^{-v_{\eps}}\nabla u_{\eps}-e^{-v_{\eps}}u_{\eps}\nabla v_{\eps}\vert^2\\
 \leq&\frac{1}{2}\int_0^T\io e^{-2v_{\eps}}\vert\nabla u_{\eps}\vert^2
 +\frac{1}{2}\int_0^T\io e^{-2v_{\eps}}u_{\eps}^2\vert\nabla v_{\eps}\vert^2\\
 \leq&\frac{1}{2}\int_0^T\io\vert\nabla u_{\eps}\vert^2
 +\frac{1}{2}\int_0^T\Vert\nabla v_{\eps}(\cdot,t)\Vert_{L^\infty(\Omega)}^2
 \Vert u_{\eps}(\cdot,t)\Vert_{L^2(\Omega)}^2\quad\text{for all }\eps\in(0,1).
\end{align*}
By Lemma \ref{elliptische Regularitätstheorie} and Lemma \ref{Linfty Schranke für u} we find $C_1>0$ such that
\begin{align*}
 \int_0^T\Vert\nabla v_{\eps}(\cdot,t)\Vert_{L^\infty(\Omega)}^2\Vert \Vert u_{\eps}(\cdot,t)\Vert_{L^2(\Omega)}^2
 \leq C_1\int_0^T\io u_{\eps}^2\quad\text{for all }\eps\in(0,1).
\end{align*}
Thus, by Lemma \ref{nabla ueps in L2L2 beschränkt} and Lemma \ref{Raum-Zeit-Beschränktheit von ueps in L2}
we find $C_2(T)>0$ such that
\begin{align*}
 \int_0^T\io\vert\nabla(u_{\eps} e^{-v_{\eps}})\vert^2\leq C(T)\quad\text{for all }\eps\in(0,1).
\end{align*}
\end{proof}

Due to the Eberlein-Smulyan theorem there is $z\in L^2(\Omega\times(0,T))$ and a subsequence
such that $\nabla(e^{-v_{\eps_{j_k}}}u_{\eps_{j_k}})
\rightharpoonup z$ in $L^2(\Omega\times(0,T))$ as $k\to\infty$. It remains to show that $z=\nabla(e^{-v}u)$.

\begin{cor}\label{schwache Konvergenz von nablae-vu in L2}
Suppose \eqref{Generalvoraussetzungen}. Let $T\in(0,\infty)$ and $(\eps_j)_{j\in\N}\subseteq(0,1)$ be a
nullsequence. Then there is a subsequence $(\eps_{j_k})_{k\in\N}\subseteq(0,1)$ such that the corresponding
classical solutions $(u_{\eps_{j_k}},v_{\eps_{j_k}})$ satisfy
 \begin{align*}
  \nabla(u_{\eps_{j_k}}e^{-v_{\eps_{j_k}}})\rightharpoonup\nabla(ue^{-v})\quad\text{in }L^2(\Omega\times(0,T))
  \quad\text{as }k\to\infty,
 \end{align*}
 where $(u,v)$ is the limit pair from Lemma \ref{Gleichmäßige Konvergenz von ueps und veps}.
\end{cor}
\begin{proof}
 Testing with $\varphi\in C_c^\infty(\Omega\times(0,T))$ yields
 \begin{align*}
  \int_0^T\io z\varphi\leftarrow\int_0^T\io \nabla(e^{-v_{\eps_{j_k}}}u_{\eps_{j_k}})\varphi
  =-\int_0^T\io e^{-v_{\eps_{j_k}}}u_{\eps_{j_k}}\nabla\cdot\varphi
 \end{align*}
Since $u_{\eps_{j_k}}\to u$ in $C^0(\overline{\Omega}\times[0,T])$ and $v_{\eps_{j_k}}\to v$ in
$C^{2,0}(\overline{\Omega}\times[0,T])$, we conclude $e^{-v_{\eps_{j_k}}}u_{\eps_{j_k}}\rightharpoonup e^{-v}u$
in $L^2(\Omega\times(0,T))$. Hence,
\begin{align*}
 -\int_0^T\io e^{-v_{\eps_{j_k}}}u_{\eps_{j_k}}\nabla\cdot\varphi\to-\int_0^T\io e^{-v}u\nabla\cdot\varphi.
\end{align*}
Thus, $z=\nabla(e^{-v}u)$ almost everywhere in $\Omega\times(0,T)$.
\end{proof}
With this convergence result, we can show that the limit functions from Lemma \ref{Gleichmäßige Konvergenz von ueps und veps}
are $W^{1,2}$-solutions of \eqref{eps=0 System} in the sense of Definition \ref{Stark schwache Lösungsdefinition}.

\begin{cor}
Suppose \eqref{Generalvoraussetzungen} and let $T\in(0,\infty)$. Then
 the limit pair $(u,v)$ obtained in Lemma \ref{Gleichmäßige Konvergenz von ueps und veps}
 is a $W^{1,2}$-solutions of \eqref{eps=0 System}.
\end{cor}
\begin{proof}
 Since $v_{\eps_{j_k}}\to v$ in $C^{2,0}(\overline{\Omega}\times[0,T])$, $u_{\eps_{j_k}}\to u$ in
 $C^0(\overline{\Omega}\times[0,T])$ as $j\to\infty$ and $(u_{\eps_{j_k}},v_{\eps_{j_k}})$ satisfies
 \begin{align*}
  0=\Delta v_{\eps_{j_k}}-v_{\eps_{j_k}}+u_{\eps_{j_k}}\text{ in }\Omega,\quad
  \partial_\nu v_{\eps_{j_k}}=0\text{ on }\partial\Omega,
 \end{align*}
it follows
\begin{align*}
 0=\Delta v_{\eps_{j_k}}-v_{\eps_{j_k}}+u_{\eps_{j_k}}
 \to\Delta v-v+u=0\quad\text{pointwise in }\Omega\times[0,T],
\end{align*}
and $\partial_\nu v=0$ on $\partial\Omega$.
Due to uniform convergence and boundedness of $u_{\eps_{j_k}}$, we have
\begin{align*}
 -\int_0^T\io u_{\eps_{j_k}}\varphi_t\to-\int_0^T\io u\varphi_t,\quad
 \eps_{j_k}\kappa\int_0^T\io u_{\eps_{j_k}}\to 0,\quad
 -\eps_{j_k}\mu\int_0^T\io u_{\eps_{j_k}}^2\to 0,
\end{align*}
as $j\to\infty$. In combination with Corollary \ref{schwache Konvergenz von nablae-vu in L2}, we obtain that $(u,v)$ is a
$W^{1,2}$-solution of \eqref{eps=0 System}.
\end{proof}

\begin{cor}\label{u,v sind klassische Lösungen}
Suppose \eqref{Generalvoraussetzungen} and let $T\in(0,\infty)$. Then
 the limit pair $(u,v)$ obtained in Lemma \ref{Gleichmäßige Konvergenz von ueps und veps}
 is a classical solution of \eqref{eps=0 System}.
\end{cor}
\begin{proof}
 Since $v$ solves the elliptic equation in \eqref{eps=0 System} classically, it has to be the unique classical solution.
 It remains to show that $u$ solves the parabolic equation in the classical sense. Due to Lemma \ref{Linfty Schranke für u}, there
 is $C_1(T)>0$ such that $\Vert u_{\eps_{j_k}}(\cdot,t)\Vert_{L^\infty(\Omega)}\leq C_1(T)$
 for all $t\in(0,T)$ and $k\in\N$.
 By \cite[Proof of Theorem 2.5]{winkler2011} we find $C_2>0$ such that
 \begin{align*}
  \Vert u_{\eps_{j_k}}\Vert_{C_{loc}^{2+\theta,1+\frac{\theta}{2}}(\overline{\Omega}\times(0,\infty))}\leq C_2
 \end{align*}
for all $k\in\N$. Then by \cite[Proof of Theorem 1.1]{FujieYokota2014} there is
$\tilde{u}\in C_{loc}^{2,1}(\overline{\Omega}\times(0,\infty))$ such that
\begin{align*}
 u_{\eps_{j_k}}\to\tilde{u}\quad\text{in }C_{loc}^{2,1}(\overline{\Omega}\times(0,\infty))
 \quad\text{as }k\to\infty,
\end{align*}
and $(\tilde{u},v)$ solves \eqref{eps=0 System} classically.
For a compact subset $[T_1,T_2]\subseteq(0,\infty)$, we have
\begin{align*}
 u_{\eps_{j_k}}\to u\quad\text{in }C^0(\overline{\Omega}\times[T_1,T_2])\quad\text{as }k\to\infty,
\end{align*}
 by Lemma \ref{Gleichmäßige Konvergenz von ueps und veps} and
\begin{align*}
 u_{\eps_{j_k}}\to\tilde{u}\quad\text{in }C^0(\overline{\Omega}\times[T_1,T_2])\quad\text{as }k\to\infty.
\end{align*}
Hence $\tilde{u}(x,t)=u(x,t)$ in $\Omega\times(0,T)$ for $T\in(0,\infty)$. Hence, the limit pair $(u,v)$ from
Lemma \ref{Gleichmäßige Konvergenz von ueps und veps} is the unique classical solution of \eqref{eps=0 System}.
\end{proof}
Due to that result, Theorem \ref{Theorem Konvergenzen} is directly proven.

\begin{proof}[Proof of Theorem \ref{Theorem Konvergenzen}]
 The uniform convergence follows from Lemma \ref{Gleichmäßige Konvergenz von ueps und veps}. That the obtained
 limit pair $(u,v)$ is the unique classical solution of \eqref{eps=0 System} follows from
Corollary \ref{u,v sind klassische Lösungen}.
\end{proof}

\section{Proof of Theorem 1.2}
Since we have uniform convergence of $u_\eps,v_\eps$ and the limit functions $u,v$ solve \eqref{eps=0 System}, we can
prove Theorem \ref{Theorem große Dichten}.
\begin{proof}[Proof of Theorem \ref{Theorem große Dichten}]
Let $u_0\in W^{1,q}(\Omega)$ for $q>2$ be as in the formulation of Theorem \ref{Theorem große Dichten}.
We then know by \cite[Theorem 1.3]{FUJIE20205338} that
\begin{align}\label{Blow-Up von u in unendlicher Zeit}
 \lim_{t\nearrow\infty}\Vert u(\cdot,t)\Vert_{L^\infty(\Omega)}=\infty,
\end{align}
where $u$ is the unique solution of the parabolic equation in \eqref{eps=0 System}.
Let $M>0$. Then there is $T(M)>0$ such that
\begin{align}
 \Vert u(\cdot,t)\Vert_{L^\infty(\Omega)}>2M
\end{align}
for all $t\geq T(M)$. Hence, there exists $x_0\in\Omega$ such that
\begin{align*}
 u(x_0,T(M))>2M.
\end{align*}
Due to continuity there is $t_0\in(0,T(M))$ such that
\begin{align*}
 u(x_0,t_0)>M.
\end{align*}
Let us assume that there is a nullsequence $(\eps_j)_{j\in\N}\subseteq(0,1)$
such that
\begin{align*}
 u_{\eps_j}(x,t)\leq M
\end{align*}
for all $(x,t)\in\Omega\times(0,T(M))$ and $j\in\N$, where $u_{\eps_j}$
is the unique classical solution of \eqref{Hauptsystem} for the corresponding $\eps_j$.
By Lemma \ref{Gleichmäßige Konvergenz von ueps und veps} there is a subsequence
$(\eps_{j_k})_{k\in\N}\subseteq(0,1)$ such that
the corresponding classical solutions $(u_{\eps_{j_k}},v_{\eps_{j_k}})$ satisfy
\begin{align*}
 &u_{\eps_{j_k}}\to\tilde{u}\quad\text{in }C^0(\overline{\Omega}\times[0,T(M)])\quad\text{as }k\to\infty,\\
 &v_{\eps_{j_k}}\to\tilde{v}\quad\text{in }C^{2,0}(\overline{\Omega}\times[0,T(M)])\quad\text{as }k\to\infty,
\end{align*}
where $(\tilde{u},\tilde{v})$ is a $W^{1,2}$-solution of \eqref{eps=0 System}. Moreover,
$\tilde{u}\in C^{\alpha,\frac{\alpha}{2}}(\overline{\Omega}\times[0,T(M)])$
for $\alpha<\lambda$ by Corollary \ref{Grenzfunktionen sind Hölderstetig}.
By Corollary \ref{u,v sind klassische Lösungen} and due to uniqueness of classical solutions of
\eqref{eps=0 System}, we conclude $\tilde{u}(x,t)=u(x,t)$ and $\tilde{v}(x,t)=v(x,t)$
for all $(x,t)\in\overline{\Omega}\times[0,T(M)]$.
Due to uniform convergence it follows for all $(x,t)\in\Omega\times(0,T(M))$
\begin{align*}
 u(x,t)=\tilde{u}(x,t)=\lim_{k\to\infty} u_{\eps_{j_k}}(x,t)\leq M.
\end{align*}
This implies
\begin{align*}
 M<u(x_0,t_0)=\tilde{u}(x_0,t_0)=\lim_{k\to\infty}u_{\eps_{j_k}}(x_0,t_0)\leq M.
\end{align*}
\end{proof}

\section{Acknowledgment}
This research was funded by the Deutsche Forschungsgemeinschaft (DFG, German Research Foundation)
– project no. 552796953 (``Explosions and solvability in cross-diffusive systems of chemotaxis type''),
whose support is gratefully acknowledged.

\phantomsection
\addcontentsline{toc}{section}{References}
\bibliographystyle{abbrv}
 \bibliography{lit.bib}

\end{document}